\documentclass[final,3p]{elsarticle}
\usepackage[T1]{fontenc}
\usepackage{amsmath,amsthm,mathtools}
\usepackage{txfonts}
\usepackage{microtype}
\usepackage{booktabs}
\usepackage{placeins}
\usepackage[hidelinks]{hyperref}
\biboptions{sort&compress}

\newcommand{\R}{\mathbb{R}}
\newcommand{\E}{\mathbb{E}}
\newcommand{\KL}{\operatorname{KL}}
\newcommand{\Cov}{\operatorname{Cov}}
\newcommand{\Var}{\operatorname{Var}}
\newcommand{\tr}{\operatorname{tr}}
\newcommand{\supp}{\operatorname{supp}}
\newcommand{\conv}{\operatorname{conv}}
\newcommand{\aff}{\operatorname{aff}}
\newcommand{\interior}{\operatorname{int}}
\newcommand{\Prox}{\operatorname{Prox}}
\newcommand{\calR}{\mathcal{R}}
\newcommand{\calE}{\mathcal{E}}
\newcommand{\norm}[1]{\lVert#1\rVert}
\newtheorem{theorem}{Theorem}
\newtheorem{lemma}[theorem]{Lemma}
\newtheorem{proposition}[theorem]{Proposition}
\newtheorem{corollary}[theorem]{Corollary}
\journal{Applied Mathematics Letters}
\makeatletter
\def\ps@pprintTitle{\let\@oddhead\@empty\let\@evenhead\@empty
  \def\@oddfoot{\hfil\thepage\hfil}\let\@evenfoot\@oddfoot}
\makeatother

\begin{document}
\begin{frontmatter}
\title{Can one regularizer reproduce posterior means at two noise levels?}
\author[xjtu]{Kang Liu\corref{cor1}}
\ead{kanyo@foxmail.com}
\cortext[cor1]{Corresponding author.}

\affiliation[xjtu]{
  organization={School of Future Technology, Xi'an Jiaotong University},
  city={Xi'an},
  country={China}
}

\begin{abstract}
Consider Gaussian denoising with a fixed prior and varying noise
variance. When the prior has a density, maximum a posteriori
estimation uses its negative logarithm as a fixed regularizer;
only the weight of the quadratic data term changes with the noise
variance. We ask whether a fixed regularizer can also recover the
posterior mean. Although a suitable regularizer exists at each
variance, we prove that no single regularizer works at two distinct
positive variances for any compactly supported prior other than a
point mass. This holds even for nonconvex and nonsmooth regularizers.
A residual computed from two denoiser evaluations gives positive
lower bounds on the uniform approximation error and excess mean
squared error of estimators using a common regularizer. Numerical
examples illustrate these bounds and the cost of reusing a
regularizer across noise levels.
\end{abstract}
\end{frontmatter}

\section{Introduction}

In Gaussian denoising, the prior describes the signal distribution,
while the noise variance determines how strongly an estimate should
follow the observation. When the prior has a density, maximum a
posteriori estimation minimizes its negative logarithm plus a
quadratic data term. Keeping the prior fixed therefore keeps the
regularizer fixed; only the weight of the data term changes with
the noise variance. The posterior mean instead minimizes mean
squared error. Can it be obtained using one regularizer across
noise levels, with the variance entering only through the weight
of the quadratic data term?

At each fixed variance, the posterior mean admits such a variational
representation, although the regularizer generally differs from
the negative log prior. Gribonval \cite{gribonval2011should}
established this representation for Gaussian denoising, and
Gribonval and Machart \cite{gribonval2013reconciling} extended it
to linear inverse problems with colored Gaussian noise.
Posterior estimation has also been studied for total variation
priors by Louchet and Moisan \cite{louchet2013posterior} and for
sparse stochastic processes by Amini et al.\ \cite{amini2013bayesian}.
Gribonval and Nikolova \cite{gribonval2020characterization}
characterize proximity operators of possibly nonconvex penalties,
and their Bayesian analysis \cite{gribonval2021bayesian} explicitly
allows the representing penalty to depend on the variance.
For log-concave priors, Darbon and Langlois
\cite{darbon2021posterior} describe matched regularizers through
Hamilton--Jacobi equations. These results establish a representation
at each variance, but leave open whether one regularizer can
represent the posterior mean at two prescribed variances.

Cross-variance reuse has been investigated by Nguyen et al.\
\cite{nguyen2018learning}, who use proximal scaling to apply learned
convex regularizers at other noise levels and report close agreement
with Bayesian denoising in their process models. The question here
is whether this agreement can be exact at every observation.
It is also relevant to reconstruction methods that use denoisers
in proximal iterations \cite{venkatakrishnan2013plug} or construct
regularization objectives from them \cite{romano2017little}.
Reehorst and Schniter \cite{reehorst2019red} identify Jacobian
symmetry as a necessary condition for a denoiser-derived vector
field to be a gradient. Gaussian posterior means already have
symmetric Jacobians through their relation to the marginal score
\cite{efron2011tweedie,raphan2011least}; compatibility across
variances imposes a further requirement.
Other connections between Bayesian and variational estimation
include Bregman losses for which MAP estimates are Bayes optimal
\cite{burger2014maximum}, learned proximal operators associated
with negative log densities \cite{fang2024prior,fang2025beyond},
and asymptotic matching through prior pairs
\cite{okudo2026matching}. In the present problem, both the prior
and the squared error loss remain fixed.

We prove that, for every compactly supported prior other than a
point mass, no single regularizer reproduces the posterior mean
at two distinct positive variances. The result allows nonconvex
and nonsmooth regularizers and follows from a comparison of their
canonical representations at a fixed posterior mean. A residual
computed from two denoiser evaluations quantifies the incompatibility,
giving positive lower bounds on uniform approximation error and
excess mean squared error for estimators using a common regularizer.
Numerical examples illustrate these bounds and the tradeoff when
one regularizer is used at both variances.

\section{Incompatibility at two noise variances}

Let $\mu$ be a probability distribution on $\R^d$, where $d\geq1$,
and consider
\[
X\sim\mu,
~~
Y_t=X+\sqrt{t}\,Z,
~~
Z\sim N(0,I_d).
\]
Here $Z$ is independent of $X$, $I_d$ is the $d\times d$ identity
matrix, and $t>0$ is the noise variance. The prior $\mu$ remains
fixed as $t$ varies. Throughout these two sections, its support
$K:=\supp\mu$ is compact.

We use the Euclidean norm $\norm{\cdot}$ and inner product
$\langle\cdot,\cdot\rangle$. The symbols $\conv K$ and $\aff K$
denote the convex and affine hulls of $K$, respectively, and
$\interior$ denotes interior in the ambient Euclidean space.
The Gaussian observation model defines the posterior mean at
every input $y\in\R^d$:
\begin{equation}
 D_t(y):=
 \frac{\int_K x\exp[-\norm{x-y}^2/(2t)]\,\mu(dx)}
 {\int_K\exp[-\norm{x-y}^2/(2t)]\,\mu(dx)}.
 \label{eq:denoiser}
\end{equation}
For a proper lower semicontinuous function
$R:\R^d\to\R\cup\{+\infty\}$, define
\begin{equation}
 \Prox_{tR}(y):=
 \mathop{\rm argmin}_{x\in\R^d}
 \left\{R(x)+\frac{\norm{x-y}^2}{2t}\right\}.
 \label{eq:prox}
\end{equation}
Properness means that $R$ is finite at some point. The set
$\Prox_{tR}(y)$ consists of all global minimizers; convexity of
$R$ is not assumed. For convex $R$, this is the classical
proximity operator \cite{moreau1965proximite}.

At a fixed variance, the posterior mean determines its representing
regularizer up to an additive constant on the range of the denoiser.
We therefore compare these regularizers at the same posterior mean
to determine whether one regularizer can work at two variances.
First suppose $\aff K=\R^d$. Define the log partition function $A_t$
and the open convex mean domain $\Omega$ by
\begin{equation}
 A_t(\theta):=\log\int_K e^{\theta^\top x-\norm{x}^2/(2t)}\,\mu(dx),
 ~~ \Omega:=\interior(\conv K),~ \theta\in\R^d.
 \label{eq:partition}
\end{equation}
The mean parametrization of a minimal exponential family
\cite[Proposition~3.1 and Theorem~3.3]{wainwright2008graphical}
gives a bijection $\nabla A_t:\R^d\to\Omega$.
The Hessian $\nabla^2A_t(\theta)$ is the covariance matrix under the
distribution proportional to $e^{\theta^\top x-\norm{x}^2/(2t)}\mu(dx)$.
It is positive definite because $K$ is not contained in a proper affine
subspace. This and the implicit function theorem \cite[Theorem~9.28]{rudin1976principles} give a smooth inverse
$\theta_t(z):=(\nabla A_t)^{-1}(z)$ for $z\in\Omega$.
Define the probability measure with mean $z$ and its total variance by
\begin{equation}
 \nu_{t,z}(dx):=e^{\theta_t(z)^\top x-\norm{x}^2/(2t)-A_t(\theta_t(z))}\mu(dx),
 ~~ V_t(z):=\E_{\nu_{t,z}}\norm{X-z}^2.
 \label{eq:tilt}
\end{equation}
Here $\E_{\nu_{t,z}}$ denotes expectation under $\nu_{t,z}$.
By construction, $\E_{\nu_{t,z}}X=z$ and $D_t(t\theta_t(z))=z$.
For a fixed $z\in\Omega$, these distributions have the same posterior
mean but correspond to different observations $t\theta_t(z)$.
The comparison below holds at this common output coordinate.
Write $A_t^*(z):=\sup_{\theta\in\R^d}
\{\theta^\top z-A_t(\theta)\}$ for the convex conjugate.
Relative entropy is $\KL(\nu\Vert\eta):=\int\log(d\nu/d\eta)\,d\nu$
when $\nu$ is absolutely continuous with respect to $\eta$, and
$+\infty$ otherwise; $d\nu/d\eta$ is the Radon--Nikodym derivative.

\begin{lemma}[Canonical regularizers and their difference]
\label{lem:canonical}
Suppose that $\aff K=\R^d$. For each $t>0$, define
\[
 \calR_t(z):=A_t^*(z)-\frac{\norm{z}^2}{2t},
 ~~ z\in\R^d.
\]
Then $\calR_t$ is proper and lower semicontinuous, and
\[
 \Prox_{t\calR_t}(y)=\{D_t(y)\}
 ~~\text{for every }y\in\R^d.
\]
Every proper function $R$ satisfying
$D_t(y)\in\Prox_{tR}(y)$ for all $y\in\R^d$ obeys
$R=\calR_t+c_t$ on $\Omega$, where $c_t\in\R$ is a constant.

For $0<s<t$, define $c_{s,t}:=1/(2s)-1/(2t)>0$.
At every $z\in\Omega$,
\begin{equation}
 \begin{aligned}
 \calR_s(z)-\calR_t(z)
 &=c_{s,t}V_s(z)+\KL(\nu_{s,z}\Vert\nu_{t,z})\\
 &=c_{s,t}V_t(z)-\KL(\nu_{t,z}\Vert\nu_{s,z}).
 \end{aligned}
 \label{eq:kl}
\end{equation}
Consequently,
\begin{equation}
 0<c_{s,t}V_s(z)
 \leq\calR_s(z)-\calR_t(z)
 \leq c_{s,t}V_t(z).
 \label{eq:sandwich}
\end{equation}
\end{lemma}
\begin{proof}[Proof sketch]
For $R=\calR_t$, the proximal objective is
$A_t^*(z)-y^\top z/t+\norm y^2/(2t)$.
Fenchel equality \cite[Theorem~23.5]{rockafellar1970convex} gives the
unique minimizer $D_t(y)$.
Any other representation satisfies
$\nabla(R+\norm{\cdot}^2/(2t))=\theta_t=\nabla A_t^*$ on $\Omega$,
by comparing global minimizers at neighboring mean coordinates
\cite[Theorem~3(b)]{gribonval2020characterization}.
Thus it differs from $\calR_t$ by a constant.
Substituting the tilted densities in relative entropy and using their
common mean gives
$\KL(\nu_{s,z}\Vert\nu_{t,z})=\calR_s(z)-\calR_t(z)-c_{s,t}V_s(z)$.
Reversing the measures proves the other identity; nonnegative relative
entropy and positive variance give the bounds.
~\ref{sup:mainproofs} provides the full argument.
\end{proof}

\begin{theorem}[No common regularizer at two variances]
\label{thm:incompatibility}
If $\mu$ has compact support and is not a point mass, then for every
$0<s<t$ no proper lower semicontinuous $R$ satisfies
\begin{equation}
 D_s(y)\in\Prox_{sR}(y),~~ D_t(y)\in\Prox_{tR}(y)
 ~~\text{for all }y\in\R^d.
 \label{eq:common}
\end{equation}
\end{theorem}
\begin{proof}
Suppose first that $\aff K=\R^d$.
By Lemma~\ref{lem:canonical}, a common regularizer would make
$\calR_s-\calR_t$ constant on $\Omega$.
This difference is strictly positive by Lemma~\ref{lem:canonical}.

Let $M:=\max_{x\in K}\norm{x}$, and choose $a\in K$ with
$\norm{a}=M$.
Fix $z_0\in\Omega$ and define
$z_\lambda:=(1-\lambda)z_0+\lambda a$ for $0\le\lambda<1$.
These points belong to $\Omega$
\cite[Theorem~6.1]{rockafellar1970convex}.
Since $\nu_{t,z_\lambda}$ is supported on $K$ and has mean
$z_\lambda$,
\[
 V_t(z_\lambda)
 =\E_{\nu_{t,z_\lambda}}\norm{X}^2-\norm{z_\lambda}^2
 \leq M^2-\norm{z_\lambda}^2
 \longrightarrow0
 ~~\text{as }\lambda\uparrow1.
\]
Applying Lemma~\ref{lem:canonical} yields
\begin{equation}
 0<\calR_s(z_\lambda)-\calR_t(z_\lambda)
 \leq c_{s,t}(M^2-\norm{z_\lambda}^2)
 \longrightarrow0.
 \label{eq:boundary}
\end{equation}
A strictly positive constant cannot have this limit, giving
the contradiction.

For a lower-dimensional support, let
$r:=\dim(\aff K)$ and write $\aff K=c+U\R^r$.
Here $c\in\aff K$ and the columns of
$U\in\R^{d\times r}$ form an orthonormal basis for the direction
space of $\aff K$, so $U^\top U=I_r$.
Let $\widetilde\mu$ be the distribution of $U^\top(X-c)$,
and let $\widetilde D_u$ be its posterior mean at noise variance
$u>0$. The Gaussian kernel gives
$D_u(c+Uy)=c+U\widetilde D_u(y)$ for $y\in\R^r$.
If a common $R$ existed, its restriction
$\widetilde R(z):=R(c+Uz)$ would satisfy
$\widetilde D_u(y)\in\Prox_{u\widetilde R}(y)$ for $y\in\R^r$ and $u\in\{s,t\}$,
by restricting the competing points in \eqref{eq:prox} to
$\aff K$.
The restriction is lower semicontinuous and is finite at the
posterior means, hence proper.
Since $\mu$ is not a point mass, $r\geq1$.
The induced prior has compact support with affine hull $\R^r$,
so the preceding argument applies.
\end{proof}

A point mass at $a$ admits the regularizer that is zero at $a$ and
$+\infty$ elsewhere. For $N(m,\Sigma)$ with mean $m$ and positive
definite covariance $\Sigma$, the quadratic
$R(z)=\tfrac12(z-m)^\top\Sigma^{-1}(z-m)$ works at every variance.
~\ref{sup:gaussian} covers singular Gaussians and characterizes
Gaussian priors by a common regularizer on a variance set with a
positive finite accumulation point.

\section{Quantitative consequences of incompatibility}
Fix $0<s<t$ and $\alpha:=1-s/t$.
For a common regularizer $R$, let $P_u(y)\in\Prox_{uR}(y)$ select a
global minimizer at every input, for $u\in\{s,t\}$.
The selections satisfy $P_s((s/t)y+\alpha P_t(y))=P_t(y)$:
a minimizer at the larger variance remains the unique minimizer at
the smaller variance after moving the observation toward it.
This compatibility relation underlies proximal scaling
\cite[Proposition~1]{nguyen2018learning} and the convex resolvent
identity \cite[Chapter~23]{bauschke2017convex}. Its violation for
the posterior means defines the residual
\begin{equation}
 C_{s,t}(y):=D_s((s/t)y+\alpha D_t(y))-D_t(y),
 \label{eq:residual}
\end{equation}
which requires one evaluation of each denoiser.

\begin{theorem}[Uniform approximation lower bound]
\label{thm:uniform}
Let $D_s$ be $L_s$-Lipschitz, so $\norm{D_s(x)-D_s(y)}\le L_s\norm{x-y}$ for all inputs. For any common proper lower
semicontinuous $R$ and selections $P_u(y)\in\Prox_{uR}(y)$ defined
for all $y\in\R^d$, $u\in\{s,t\}$, set
$\varepsilon_u:=\sup_y\norm{P_u(y)-D_u(y)}$. Then
\begin{equation}
 \norm{C_{s,t}(y)}\le\varepsilon_s+(1+\alpha L_s)\varepsilon_t,
 ~~
 \max\{\varepsilon_s,\varepsilon_t\}
 \ge\frac{\norm{C_{s,t}(y)}}{2+\alpha L_s}.
 \label{eq:uniform}
\end{equation}
For a compact prior other than a point mass, $C_{s,t}$ is not identically zero and
one may take $L_s=B^2/(4s)$, where
$B:=\operatorname{diam}K=\sup_{x,z\in K}\norm{x-z}$.
\end{theorem}
\begin{proof}[Proof sketch]
Set $z:=P_t(y)$, $w:=(s/t)y+\alpha z$, and
$Q_u(x;q):=R(x)+\norm{x-q}^2/(2u)$. Expansion gives
\begin{equation}
 Q_s(x;w)-Q_s(z;w)
 =Q_t(x;y)-Q_t(z;y)+c_{s,t}\norm{x-z}^2.
 \label{eq:square}
\end{equation}
Thus $z$ is the unique global minimizer at $(s,w)$, proving
compatibility even for nonconvex $R$.
For $w_P:=(s/t)y+\alpha P_t(y)$, the triangle inequality and
Lipschitz continuity give
\begin{equation}
 \norm{C_{s,t}(y)}\le\norm{D_s(w_P)-P_s(w_P)}
 +(1+\alpha L_s)\norm{P_t(y)-D_t(y)},
 \label{eq:local}
\end{equation}
which proves \eqref{eq:uniform}.
If $C_{s,t}\equiv0$ and $\aff K=\R^d$, take $y=t\theta_t(z)$.
Injectivity of $D_s$ gives $s\theta_s(z)=s\theta_t(z)+\alpha z$,
so $\nabla\calR_s(z)=\nabla\calR_t(z)$, contradicting
\eqref{eq:boundary}. Lower-dimensional supports reduce to their
affine hull as above.
Finally, the Jacobian satisfies
$JD_s(y)=\Cov(X\mid Y_s=y)/s$, where $\Cov$ denotes covariance.
Each unit projection of $X$ has range of length at most $B$,
so its posterior variance is at most $B^2/4$
\cite{bhatia2000variance}. This bounds the Jacobian operator norm
and gives $L_s=B^2/(4s)$.~\ref{sup:mainproofs} supplies the detailed proof.
\end{proof}

For measurable $P_u$, define the excess mean squared error
$\calE_u(P_u):=\E\norm{P_u(Y_u)-X}^2-\E\norm{D_u(Y_u)-X}^2$.
The conditional expectation projection identity
\cite[Section~9.4]{williams1991probability} gives
$\calE_u(P_u)=\E\norm{P_u(Y_u)-D_u(Y_u)}^2$.
A nonzero residual also forces positive excess risk at at least one
variance. The uniform bound alone locates an error at an individual
observation. Monotonicity of proximal selections extends that error
to a nearby set of positive probability.

\begin{corollary}[Positive excess risk]
\label{cor:risk}
Let $K$ contain more than one point and satisfy $K\subseteq\{x:\norm x\le M\}$ for some $M>0$.
Choose $y_0$ with $\delta:=\norm{C_{s,t}(y_0)}>0$ and define
\begin{equation}
 L:=\frac{B^2}{4s},~~ a:=\frac{\delta}{2+\alpha L},~~
 b:=\max\{\norm{y_0},\norm{(s/t)y_0+\alpha D_t(y_0)}+\alpha a\}.
 \label{eq:riskconstants}
\end{equation}
Here $L$ bounds the Lipschitz constants of both denoisers, while $a$
and $b$ bound the error size and the norm of an input where it occurs.
For measurable selections from any common $R$ as in
Theorem~\ref{thm:uniform},
\begin{equation}
 \max\{\calE_s(P_s),\calE_t(P_t)\}
 \ge\frac{\omega_d a^{d+2}}{16^{d+1}L^d(2\pi t)^{d/2}}
 \exp\!\left[-\frac{(b+M+a/(4L))^2}{2s}\right]>0,
 \label{eq:riskbound}
\end{equation}
where $\omega_d$ is the volume of the Euclidean unit ball.
\end{corollary}
\begin{proof}[Proof sketch]
Global minimality gives $\langle P_u(x)-P_u(y),x-y\rangle\ge0$,
so each selection is monotone. By \eqref{eq:local}, an error of at
least $a$ occurs at an input of norm at most $b$. Monotonicity and
Lipschitz continuity of the denoiser give an error of at least $a/4$
on a nearby ball of radius $a/(16L)$. Integrating over that ball,
using a lower bound on the Gaussian observation density, yields
\eqref{eq:riskbound}.~\ref{sup:riskproof} gives the full proof, and
~\ref{sup:approximation} accounts for approximate
denoiser evaluations.
\end{proof}

\section{Numerical illustrations}
\label{sec:numerics}

The experiments illustrate the variation of canonical regularizers,
the approximation lower bound obtained from the residual, and the
excess risk of using one regularizer at two variances.
We consider four compact priors: equal mass at $\{-1,1\}$;
masses $(0.12,0.38,0.29,0.21)$ at $(-1.2,-0.1,0.6,1.8)$;
masses $(0.2,0.5,0.3)$ at
$\{(-1,-0.4),\allowbreak(1.1,-0.3),\allowbreak(0.2,1.3)\}$;
and the uniform distribution on $[-1,1]$.
The standard Gaussian provides a control for which a common
quadratic regularizer exists.
Discrete posteriors are evaluated by finite sums and the uniform
posterior by numerical quadrature.
~\ref{sup:numerics} gives the evaluation procedures,
accuracy checks, and full numerical tables.

\begin{figure}[ht]
\centering
\includegraphics[width=0.8\linewidth]{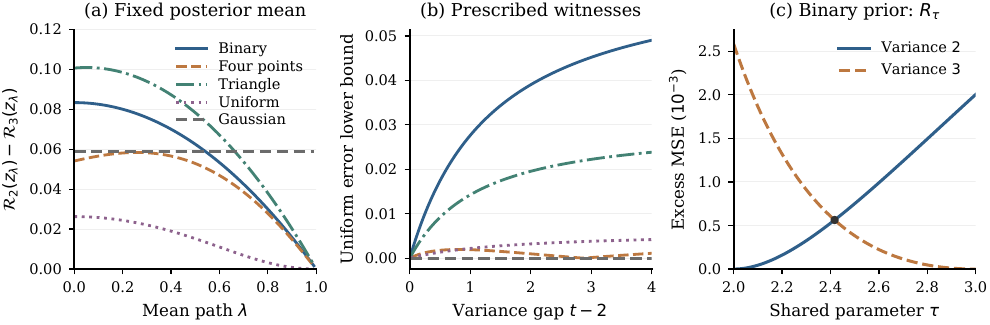}
\caption{Regularizer variation and reuse across noise levels.
(a) Canonical regularizer differences evaluated at common posterior
means, with $s=2$ and $t=3$.
(b) Lower bounds on uniform approximation error at prescribed
observations, with $s=2$ and varying $t$.
(c) Excess mean squared errors for the binary regularizer family
$R_\tau$ at variances $2$ and $3$; the dot marks approximately
equal excess risks.}
\label{fig:experiments}
\end{figure}

Figure~\ref{fig:experiments}(a) fixes $s=2$ and $t=3$.
For each compact prior, we evaluate
$\calR_s(z_\lambda)-\calR_t(z_\lambda)$ along
$z_\lambda:=(1-\lambda)\E X+\lambda a$
at $100$ values of $\lambda$ in $[0,0.99]$, where
$a\in K$ maximizes the Euclidean norm on the support $K$.
Both regularizers are evaluated at the same posterior mean
$z_\lambda$.
The computed differences are positive along these paths and
become small near the boundary, consistent with
Lemma~\ref{lem:canonical} and the boundary argument in
Theorem~\ref{thm:incompatibility}.
The Gaussian control follows $z_\lambda=\lambda$ and has
the constant difference $\tfrac12\log(9/8)$.
This constant reflects an additive normalization and does not
affect the minimizers.

Panel~(b) fixes $s=2$, varies $t$ over $[2.02,6]$, and evaluates
the lower bound
$\norm{C_{s,t}(y)}/(2+\alpha L_s)$ from
Theorem~\ref{thm:uniform}.
Here $\alpha=1-s/t$, and
$L_s=(\operatorname{diam}K)^2/(4s)$ for each compact prior.
We use the prescribed observation
$y=t\operatorname{arctanh}(1/2)v_d$, with $v_d:=(1,\ldots,1)^\top/\sqrt d$,
where $d$ is the signal dimension, $v_d$ is a unit vector,
and $\operatorname{arctanh}$ is the inverse hyperbolic tangent.
For the binary prior, this choice gives $D_t(y)=1/2$.
At $t=3$, the resulting lower bounds are
$0.027631$, $0.001910$, $0.014250$, and $0.002192$
for the binary, four point, triangular, and uniform priors,
respectively.
The uniform-prior value is a quadrature approximation.
These positive values quantify the obstruction to uniform
approximation by proximal selections of a common regularizer.
The Gaussian residual is zero analytically and remains at
rounding scale numerically.

For the binary prior, panel~(c) examines $R_\tau(z):=h(z)-z^2/(2\tau)$
for $|z|\le1$ and $R_\tau(z):=+\infty$ otherwise, with $\tau\in[2,3]$.
Here $h(z):=((1+z)\log(1+z)+(1-z)\log(1-z))/2$ on $[-1,1]$,
with $0\log0=0$.
Each $R_\tau$ is convex and reproduces the posterior mean
at variance $\tau$.
For $u\in\{2,3\}$, let $P_{u,\tau}(y)$ be the unique minimizer
of $R_\tau(z)+(z-y)^2/(2u)$.
The panel plots the excess risks $\calE_u(P_{u,\tau})$ as
$\tau$ varies.

The regularizer matched to variance $2$ incurs excess risk
$2.581\times10^{-3}$ at variance $3$; the one matched to
variance $3$ incurs $2.003\times10^{-3}$ at variance $2$.
At $\tau\approx2.416996$, the two excess risks are approximately
equal to $5.628\times10^{-4}$.
For the same variance pair and the binary observation used in
panel~(b), Corollary~\ref{cor:risk} gives a lower bound of
$1.292\times10^{-8}$ on the larger excess risk of any pair of
proximal selections generated by a common regularizer.
The computed risks in this family are substantially larger
than this bound.

\bibliographystyle{elsarticle-num}
\bibliography{references}

\clearpage
\appendix
\numberwithin{theorem}{section}

\section{Notation and the range of the posterior mean}
\label{sup:notation}
Let $\mu$ be a probability measure on $\R^d$, let $X\sim\mu$, and let
$Z\sim N(0,I_d)$ be independent of $X$. At variance $t>0$, the observation
$Y_t=X+\sqrt t Z$ has density
\begin{equation}
 p_t(y)=(2\pi t)^{-d/2}\int_{\R^d}
 e^{-\norm{x-y}^2/(2t)}\,\mu(dx).
 \label{sup:density}
\end{equation}
The associated denoiser is
\begin{equation}
 D_t(y)=\frac{\int_{\R^d}x e^{-\norm{x-y}^2/(2t)}\,\mu(dx)}
 {\int_{\R^d}e^{-\norm{x-y}^2/(2t)}\,\mu(dx)}.
 \label{sup:denoiser}
\end{equation}
The matrix $I_d$ is the $d\times d$ identity matrix. The symbol $\interior$ denotes ambient interior. The norm and inner product are Euclidean, and
$\mathbb B(q,r)=\{x:\norm{x-q}\leq r\}$ denotes a closed ball. For a set $C$, $\conv C$
and $\aff C$ denote its convex and affine hulls, and $\iota_C$ is
zero on $C$ and $+\infty$ elsewhere. We write $\Cov$ for covariance,
$\Var$ for scalar variance, and $\tr$ for matrix trace.
The denominator is positive, and Gaussian weighting makes the numerator
finite at every $(t,y)$. If $X$ is integrable, this formula specifies a
version of $\E[X\mid Y_t=y]$; conditional expectation is understood as
in \cite[Chapter~9]{williams1991probability}. Formula
\eqref{sup:denoiser} also defines the denoiser for general probability
measures.

For a proper lower semicontinuous function
$R:\R^d\to\R\cup\{+\infty\}$, define
\begin{equation}
 \Prox_{tR}(y)=\mathop{\rm argmin}_{x\in\R^d}
 \left\{R(x)+\frac{\norm{x-y}^2}{2t}\right\}.
 \label{sup:prox}
\end{equation}
A selection $P_t(y)\in\Prox_{tR}(y)$ therefore chooses a global minimizer.
We call $R$ common to a collection of denoisers when this same function
satisfies $D_t(y)\in\Prox_{tR}(y)$ at every specified variance and input.

For the rest of this section and Section~\ref{sup:variation}, assume that
$K=\supp\mu$ is compact and $\aff K=\R^d$. Put
\begin{equation}
 \begin{split}
 A_t(\theta)&=\log\int_K
 e^{\theta^\top x-\norm{x}^2/(2t)}\,\mu(dx),~~
 \Omega=\interior\conv K,\\
 \widetilde\nu_{t,\theta}(dx)&=
 e^{\theta^\top x-\norm{x}^2/(2t)-A_t(\theta)}\,\mu(dx).
 \end{split}
 \label{sup:partition}
\end{equation}
Differentiation under the integral gives the usual exponential family
identities \cite{wainwright2008graphical}:
\begin{equation}
 \nabla A_t(\theta)=\E_{\widetilde\nu_{t,\theta}}X,~~
 \nabla^2 A_t(\theta)=\Cov_{\widetilde\nu_{t,\theta}}(X)\succ0,
 ~~ D_t(y)=\nabla A_t(y/t).
 \label{sup:moments}
\end{equation}
The strict inequality follows because the tilt has a positive density
with respect to $\mu$ and hence the same support. A zero variance in a
nonzero direction would place $K$ in a proper affine hyperplane.

\begin{proposition}[Mean parametrization]
\label{sup:meanrange}
The map $\nabla A_t$ is a smooth bijection from $\R^d$ onto $\Omega$.
Its inverse $\theta_t(z):=(\nabla A_t)^{-1}(z)$ is smooth jointly in $(t,z)\in(0,\infty)\times\Omega$.
\end{proposition}
\begin{proof}
The posterior mean belongs to $\conv K$. If it lay on its boundary, a
supporting hyperplane would give a nonnegative affine function of $X$
with zero posterior expectation. That function would vanish on the
whole support, contradicting $\aff K=\R^d$. Thus the range is contained
in $\Omega$. The positive definite Hessian makes $A_t$ strictly convex,
so its gradient is injective.

Fix $z\in\Omega$ and choose $\eta>0$ such that
$\mathbb B(z,3\eta)\subset\conv K$. For every unit vector $v$, there is a point
$x_v\in K$ with $v^\top(x_v-z)>2\eta$. Choose a neighborhood $U_v$ of
$v$ on the unit sphere and a neighborhood $N_v$ of $x_v$ such that
\[
 u^\top(x-z)\ge\eta~~(u\in U_v,\ x\in N_v\cap K).
\]
Each $N_v$ has positive $\mu$ measure. A finite collection
$U_{v_1},\ldots,U_{v_m}$ covers the unit sphere, and
\[
 c_t:=\min_{1\le j\le m}\int_{N_{v_j}}
 e^{-\norm{x}^2/(2t)}\,\mu(dx)>0.
\]
For every unit vector $u$ and $r\ge0$, selecting a covering neighborhood
gives
\begin{equation}
 A_t(ru)-ru^\top z\ge r\eta+\log c_t.
 \label{sup:coercivity}
\end{equation}
Consequently $A_t(\theta)-\theta^\top z$ is coercive. Its unique minimizer
$\theta_t(z)$ satisfies $\nabla A_t(\theta_t(z))=z$, which proves
surjectivity. The Hessian in \eqref{sup:moments} is invertible, so the
implicit function theorem \cite[Theorem~9.28]{rudin1976principles} gives a continuously differentiable inverse;
iteration of its derivative formula gives the asserted smoothness.
\end{proof}

Define
\begin{equation}
 \nu_{t,z}:=\widetilde\nu_{t,\theta_t(z)},~~
 V_t(z)=\E_{\nu_{t,z}}\norm{X-z}^2,~~
 \calR_t(z)=A_t^*(z)-\frac{\norm{z}^2}{2t},
 \label{sup:canonical}
\end{equation}
where $A_t^*(z)=\sup_\theta\{\theta^\top z-A_t(\theta)\}$ is
the convex conjugate and the final expression is defined
on all of $\R^d$ as an extended real function. This is the canonical proximal representation
\cite{gribonval2011should,gribonval2020characterization}.
The conjugate is proper and lower semicontinuous. Subtracting the
continuous quadratic preserves these properties. Fenchel equality \cite{rockafellar1970convex} gives
\begin{equation}
 \calR_t(z)+\frac{\norm{z-y}^2}{2t}
 =A_t^*(z)-\frac{y^\top z}{t}+\frac{\norm y^2}{2t},
 ~~ \Prox_{t\calR_t}(y)=\{D_t(y)\}.
 \label{sup:canonicalprox}
\end{equation}
On $\Omega$ the canonical penalty is smooth, with
\begin{equation}
 \nabla\calR_t(z)=\theta_t(z)-z/t,~~
 \nabla^2\calR_t(z)=\Cov_{\nu_{t,z}}(X)^{-1}-I_d/t.
 \label{sup:canonicalderivatives}
\end{equation}

If $r:=\dim(\aff K)\ge1$ and $\aff K=c+U\R^r$, where $c\in\aff K$ and $U\in\R^{d\times r}$ has orthonormal columns, the
same construction applies to the law of $U^\top(X-c)$, whose denoiser is denoted by $\widetilde D_t$. In particular,
$D_t(c+Uy)=c+U\widetilde D_t(y)$, and its range is the relative interior
of $\conv K$. The preceding proposition applies in these intrinsic
coordinates. For a singleton support, $D_t$ is constant.

\section{How the canonical penalty changes with the variance}
\label{sup:variation}
The canonical penalty has a constrained entropy representation that
compares posterior distributions with a prescribed mean. Here $\nu\ll\mu$ means absolute continuity and $d\nu/d\mu$ is the Radon--Nikodym derivative. Relative entropy is
$\KL(\nu\Vert\mu)=\int\log(d\nu/d\mu)\,d\nu$ for
$\nu\ll\mu$, and $+\infty$ otherwise.

\begin{proposition}[Entropy representation and variance derivative]
\label{sup:variance}
Under the compact support assumptions of Section~\ref{sup:notation},
for every $z\in\Omega$ and $t>0$,
\begin{equation}
 \calR_t(z)=
 \min_{\substack{\nu\ll\mu\\\E_\nu X=z}}
 \left\{\KL(\nu\Vert\mu)
       +\frac{1}{2t}\E_\nu\norm{X-z}^2\right\}.
 \label{sup:entropy}
\end{equation}
The minimizer is uniquely $\nu_{t,z}$. At a fixed mean $z$,
\begin{equation}
 \partial_t\calR_t(z)=-\frac{V_t(z)}{2t^2},~~
 \calR_s(z)-\calR_t(z)=\int_s^t\frac{V_u(z)}{2u^2}\,du
 ~(0<s<t).
 \label{sup:timeidentity}
\end{equation}
\end{proposition}
\begin{proof}
For every feasible $\nu$, substitution of its common mean $z$ into
\eqref{sup:partition} yields
\begin{align}
 \KL(\nu\Vert\nu_{t,z})
 &=\KL(\nu\Vert\mu)-\theta_t(z)^\top z
   +\frac{\E_\nu\norm X^2}{2t}+A_t(\theta_t(z))\notag\\
 &=\KL(\nu\Vert\mu)+\frac{\E_\nu\norm{X-z}^2}{2t}-\calR_t(z).
 \label{sup:entropyidentity}
\end{align}
The logarithm of the tilted density is bounded on $K$, so this identity
also holds when the relative entropy is infinite. Nonnegativity of
relative entropy, with equality only for identical distributions,
proves \eqref{sup:entropy}. This is the entropy duality of the log
partition function \cite{wainwright2008graphical} with a mean constraint.

By Proposition~\ref{sup:meanrange}, $\theta_t(z)$ is smooth in $t$.
Differentiate
$A_t^*(z)=\theta_t(z)^\top z-A_t(\theta_t(z))$ at fixed $z$.
The terms containing $\partial_t\theta_t(z)$ cancel, and
\begin{equation}
 \partial_t A_t^*(z)
 =-\partial_t A_t(\theta_t(z))
 =-\frac{\E_{\nu_{t,z}}\norm X^2}{2t^2}.
 \label{sup:conjugatederivative}
\end{equation}
Subtracting the derivative of $\norm z^2/(2t)$ gives the first identity
in \eqref{sup:timeidentity}. The second follows by integration.
\end{proof}

The observation corresponding to the prescribed posterior mean $z$ is
$t\theta_t(z)$ and varies with $t$. Thus \eqref{sup:timeidentity}
describes the change of the penalty at a fixed output coordinate.
Testing \eqref{sup:entropy} with the minimizer at another variance gives
\begin{equation}
 \left(\frac1{2s}-\frac1{2t}\right)V_s(z)
 \le\calR_s(z)-\calR_t(z)
 \le\left(\frac1{2s}-\frac1{2t}\right)V_t(z).
 \label{sup:endpointbounds}
\end{equation}
In particular, $V_s(z)\le V_t(z)$ for $s<t$. These inequalities compare
posterior variances under the same mean constraint.

\section{Canonical comparison and the finite residual}
\label{sup:mainproofs}
This section gives the detailed proofs of the canonical comparison and
uniform approximation bound stated in the main article. Throughout,
$K=\supp\mu$ is compact. The denoisers and proximal sets are defined
in \eqref{sup:denoiser} and \eqref{sup:prox}. When $\aff K=\R^d$,
use $\Omega$, $\theta_t$, $\nu_{t,z}$, and $V_t$ from
Section~\ref{sup:notation} and relative entropy from
Section~\ref{sup:variation}.

\begin{lemma}[Canonical regularizers and their difference]
\label{sup:main:lem-canonical}
Suppose that $\aff K=\R^d$. For each $t>0$, define
\[
 \calR_t(z):=A_t^*(z)-\frac{\norm{z}^2}{2t},
 ~~ z\in\R^d.
\]
Then $\calR_t$ is proper and lower semicontinuous, and
\[
 \Prox_{t\calR_t}(y)=\{D_t(y)\}
 ~~\text{for every }y\in\R^d.
\]
Every proper function $R$ satisfying
$D_t(y)\in\Prox_{tR}(y)$ for all $y\in\R^d$ obeys
$R=\calR_t+c_t$ on $\Omega$, where $c_t\in\R$ is a constant.

For $0<s<t$, define $c_{s,t}:=1/(2s)-1/(2t)>0$.
At every $z\in\Omega$,
\begin{equation}
 \begin{aligned}
 \calR_s(z)-\calR_t(z)
 &=c_{s,t}V_s(z)+\KL(\nu_{s,z}\Vert\nu_{t,z})\\
 &=c_{s,t}V_t(z)-\KL(\nu_{t,z}\Vert\nu_{s,z}).
 \end{aligned}
 \label{sup:main:eq-kl}
\end{equation}
Consequently,
\begin{equation}
 0<c_{s,t}V_s(z)
 \leq\calR_s(z)-\calR_t(z)
 \leq c_{s,t}V_t(z).
 \label{sup:main:eq-sandwich}
\end{equation}
\end{lemma}

\begin{proof}
For $R=\calR_t$, the objective in \eqref{sup:prox} becomes
\[
 \calR_t(z)+\frac{\norm{z-y}^2}{2t}
 =
 A_t^*(z)-\frac{y^\top z}{t}+\frac{\norm{y}^2}{2t}.
\]
Fenchel equality \cite[Theorem~23.5]{rockafellar1970convex}
shows that its unique minimizer is
$z=\nabla A_t(y/t)=D_t(y)$.
Since $A_t^*$ is proper and lower semicontinuous, subtracting
the continuous function $\norm{z}^2/(2t)$ gives the same
properties for $\calR_t$.

Uniqueness up to an additive constant also follows from
\cite[Theorem~3(b)]{gribonval2020characterization}.
Here it can be proved directly.
Suppose that $R$ represents $D_t$ at every input, and put
$G(z):=R(z)+\norm{z}^2/(2t)$.
For $z,w\in\Omega$, both $R(z)$ and $R(w)$ are finite because
they are global minimizers at their corresponding inputs.
Write $h:=w-z$.
Comparing the minimizers at $t\theta_t(z)$ and $t\theta_t(w)$
gives
\begin{equation}
 \theta_t(z)^\top h
 \leq G(w)-G(z)
 \leq\theta_t(w)^\top h.
 \label{sup:main:eq-squeeze}
\end{equation}
After subtracting $\theta_t(z)^\top h$, the difference is bounded
between zero and
\[
 \norm{\theta_t(w)-\theta_t(z)}\,\norm{h}.
\]
Continuity of $\theta_t$ therefore implies that $G$ is
differentiable and
$\nabla G(z)=\theta_t(z)=\nabla A_t^*(z)$.
Hence $R-\calR_t=G-A_t^*$ is constant on the convex set $\Omega$.

To compare two variances, use \eqref{sup:partition} and the common
mean $z$ to obtain
\[
 \begin{aligned}
 \KL(\nu_{s,z}\Vert\nu_{t,z})
 &=(\theta_s(z)-\theta_t(z))^\top z
   -A_s(\theta_s(z))+A_t(\theta_t(z))\\
 &~~-c_{s,t}\E_{\nu_{s,z}}\norm{X}^2\\
 &=\calR_s(z)-\calR_t(z)-c_{s,t}V_s(z).
 \end{aligned}
\]
Reversing the two measures similarly gives
\[
 \KL(\nu_{t,z}\Vert\nu_{s,z})
 =-\calR_s(z)+\calR_t(z)+c_{s,t}V_t(z).
\]
These are the identities in \eqref{sup:main:eq-kl}.
Relative entropy is nonnegative, and $V_s(z)>0$ because
$\nu_{s,z}$ has the same full affine support as $\mu$.
This proves \eqref{sup:main:eq-sandwich}.
\end{proof}

\begin{theorem}[No common regularizer at two variances]
\label{sup:main:thm-incompatibility}
If $\mu$ has compact support and is not a point mass, then for every
$0<s<t$ there is no proper lower semicontinuous function $R$ such that
\begin{equation}
 D_s(y)\in\Prox_{sR}(y),
 ~~
 D_t(y)\in\Prox_{tR}(y)
 ~~\text{for all }y\in\R^d.
 \label{sup:main:eq-common}
\end{equation}
\end{theorem}

\begin{proof}
Suppose first that $\aff K=\R^d$.
By Lemma~\ref{sup:main:lem-canonical}, a common regularizer would make
$\calR_s-\calR_t$ constant on $\Omega$.
This difference is strictly positive by \eqref{sup:main:eq-sandwich}.

Let $M:=\max_{x\in K}\norm{x}$, and choose $a\in K$ with
$\norm{a}=M$.
Fix $z_0\in\Omega$ and define
\[
 z_\lambda:=(1-\lambda)z_0+\lambda a,
 ~~ 0\leq\lambda<1.
\]
These points belong to $\Omega$
\cite[Theorem~6.1]{rockafellar1970convex}.
Since $\nu_{t,z_\lambda}$ is supported on $K$ and has mean
$z_\lambda$,
\[
 V_t(z_\lambda)
 =\E_{\nu_{t,z_\lambda}}\norm{X}^2-\norm{z_\lambda}^2
 \leq M^2-\norm{z_\lambda}^2
 \longrightarrow0
 ~~\text{as }\lambda\uparrow1.
\]
Applying \eqref{sup:main:eq-sandwich} yields
\begin{equation}
 0<\calR_s(z_\lambda)-\calR_t(z_\lambda)
 \leq c_{s,t}(M^2-\norm{z_\lambda}^2)
 \longrightarrow0.
 \label{sup:main:eq-boundary}
\end{equation}
A strictly positive constant cannot have this limit, giving
the contradiction.

For a lower-dimensional support, let
$r:=\dim(\aff K)$ and write $\aff K=c+U\R^r$.
Here $c\in\aff K$ and the columns of
$U\in\R^{d\times r}$ form an orthonormal basis for the direction
space of $\aff K$, so $U^\top U=I_r$.
Let $\widetilde\mu$ be the distribution of $U^\top(X-c)$,
and let $\widetilde D_u$ be its posterior mean at noise variance
$u>0$. The Gaussian kernel gives
\[
 D_u(c+Uy)=c+U\widetilde D_u(y),
 ~~ y\in\R^r.
\]
If a common $R$ existed, its restriction
$\widetilde R(z):=R(c+Uz)$ would satisfy
\[
 \widetilde D_u(y)\in\Prox_{u\widetilde R}(y),
 ~~ y\in\R^r,~ u\in\{s,t\},
\]
by restricting the competing points in \eqref{sup:prox} to
$\aff K$.
The restriction is lower semicontinuous and is finite at the
posterior means, hence proper.
Since $\mu$ is not a point mass, $r\geq1$.
The induced prior has compact support with affine hull $\R^r$,
so the preceding argument applies.
\end{proof}

Fix $0<s<t$, define $\alpha:=1-s/t$, and set
\begin{equation}
C_{s,t}(y):=D_s((s/t)y+\alpha D_t(y))-D_t(y).
\label{sup:main:residual}
\end{equation}
\begin{theorem}[Uniform approximation lower bound]
\label{sup:main:thm-uniform}
Suppose that $D_s$ is $L_s$-Lipschitz, meaning that
\[
 \norm{D_s(x)-D_s(y)}\leq L_s\norm{x-y}
 ~~\text{for all }x,y\in\R^d.
\]
Let $R$ be proper and lower semicontinuous, and let
$P_u(y)\in\Prox_{uR}(y)$ be defined for every $y\in\R^d$
and $u\in\{s,t\}$.
Define the uniform errors
\[
 \varepsilon_u:=
 \sup_{y\in\R^d}\norm{P_u(y)-D_u(y)},
 ~~ u\in\{s,t\}.
\]
Then, for every $y\in\R^d$,
\begin{equation}
 \begin{aligned}
 \norm{C_{s,t}(y)}
 &\leq\varepsilon_s+(1+\alpha L_s)\varepsilon_t,\\
 \max\{\varepsilon_s,\varepsilon_t\}
 &\geq\frac{\norm{C_{s,t}(y)}}{2+\alpha L_s}.
 \end{aligned}
 \label{sup:main:eq-uniform}
\end{equation}
If $\mu$ is not a point mass, then $C_{s,t}$ is not identically
zero. One may take $L_s=B^2/(4s)$, where
\[
 B:=\operatorname{diam}K
   =\sup_{x,z\in K}\norm{x-z}
\]
is the diameter of the support.
\end{theorem}

\begin{proof}
Fix $y\in\R^d$, and put
$z:=P_t(y)$ and $w:=(s/t)y+\alpha z$.
For $u>0$ and $x,q\in\R^d$, write
\[
 Q_u(x;q):=R(x)+\frac{\norm{x-q}^2}{2u}
\]
for the proximal objective.
Expanding the quadratic terms gives
\begin{equation}
 Q_s(x;w)-Q_s(z;w)
 =
 Q_t(x;y)-Q_t(z;y)+c_{s,t}\norm{x-z}^2,
 \label{sup:main:eq-square}
\end{equation}
where $c_{s,t}=1/(2s)-1/(2t)>0$.
Since $z$ minimizes $Q_t(\cdot;y)$, it is the unique global
minimizer of $Q_s(\cdot;w)$.
Thus
$P_s((s/t)y+\alpha P_t(y))=P_t(y)$.

Define the two inputs
\[
 w_D:=\frac{s}{t}y+\alpha D_t(y),
 ~~
 w_P:=\frac{s}{t}y+\alpha P_t(y).
\]
Using $P_s(w_P)=P_t(y)$ and the Lipschitz property of $D_s$,
we obtain
\begin{equation}
 \norm{C_{s,t}(y)}
 \leq
 \norm{D_s(w_P)-P_s(w_P)}
 +(1+\alpha L_s)\norm{P_t(y)-D_t(y)}.
 \label{sup:main:eq-local}
\end{equation}
Bounding the two errors by $\varepsilon_s$ and $\varepsilon_t$
proves \eqref{sup:main:eq-uniform}.

To prove that the residual is nonzero somewhere, first suppose
that $\aff K=\R^d$.
If $C_{s,t}$ were identically zero, substitution of
$y=t\theta_t(z)$ into \eqref{sup:main:residual} would give
\[
 D_s(s\theta_t(z)+\alpha z)=z
 ~~\text{for every }z\in\Omega.
\]
The map $D_s$ is injective because
$D_s(y)=\nabla A_s(y/s)$.
Since $D_s(s\theta_s(z))=z$, it follows that
\[
 s\theta_s(z)=s\theta_t(z)+\alpha z.
\]
Consequently,
\[
 \nabla\calR_s(z)
 =\theta_s(z)-\frac{z}{s}
 =\theta_t(z)-\frac{z}{t}
 =\nabla\calR_t(z).
\]
The difference $\calR_s-\calR_t$ would therefore be constant
on $\Omega$, contradicting \eqref{sup:main:eq-boundary}.
For a lower-dimensional support, the same argument applies
after the affine reduction used in
Theorem~\ref{sup:main:thm-incompatibility}.

Finally, differentiation of \eqref{sup:denoiser} gives
\[
 JD_s(y)=\frac1s\Cov(X\mid Y_s=y),
\]
where $JD_s(y)$ is the Jacobian matrix of $D_s$.
For any scalar random variable $Q$ taking values in an interval
$[\ell,r]$, nonnegativity of $\E[(Q-\ell)(r-Q)]$ gives
\[
 \Var(Q)
 \leq(\E Q-\ell)(r-\E Q)
 \leq\frac{(r-\ell)^2}{4},
\]
where $\Var$ denotes scalar variance
\cite{bhatia2000variance}.
For every unit vector $v\in\R^d$, the projection $v^\top X$
takes values in an interval of length at most $B$.
Its conditional variance is therefore at most $B^2/4$.
Hence the operator norm of $JD_s(y)$ is at most $B^2/(4s)$
at every input, proving the stated Lipschitz bound.
\end{proof}

\section{Certificates from approximate denoisers}
\label{sup:approximation}
Fix $0<s<t$, put $\alpha=1-s/t$, and suppose $D_s$ is globally
$L_s$-Lipschitz. Given maps $\widehat D_s,\widehat D_t:\R^d\to\R^d$,
define
\begin{equation}
 \begin{split}
 C_{s,t}(y)&=D_s\big((s/t)y+\alpha D_t(y)\big)-D_t(y),\\
 \widehat C_{s,t}(y)&=
 \widehat D_s\big((s/t)y+\alpha\widehat D_t(y)\big)-\widehat D_t(y).
 \end{split}
 \label{sup:residuals}
\end{equation}
Here $\norm{F}_\infty=\sup_{y\in\R^d}\norm{F(y)}$.

\begin{proposition}[Correction for approximation error]
\label{sup:approximatecertificate}
If $\norm{\widehat D_u-D_u}_\infty\le\eta_u$ for $u\in\{s,t\}$, then
\begin{equation}
 \norm{\widehat C_{s,t}-C_{s,t}}_\infty
 \le\eta_s+(1+\alpha L_s)\eta_t.
 \label{sup:residualerror}
\end{equation}
For any selections $P_u(y)\in\Prox_{uR}(y)$ defined at all inputs and
both variances from a common proper lower semicontinuous $R$,
\begin{equation}
 \max_{u\in\{s,t\}}\norm{P_u-D_u}_\infty
 \ge\frac{
 [\norm{\widehat C_{s,t}(y)}-\eta_s-(1+\alpha L_s)\eta_t]_+}
 {2+\alpha L_s}
 ~~(y\in\R^d),
 \label{sup:correctedcertificate}
\end{equation}
where $[q]_+=\max\{q,0\}$.
\end{proposition}
\begin{proof}
Set $w=(s/t)y+\alpha D_t(y)$ and
$\widehat w=(s/t)y+\alpha\widehat D_t(y)$. Adding and subtracting
$D_s(\widehat w)$ gives
\begin{align*}
 \norm{\widehat C_{s,t}(y)-C_{s,t}(y)}
 &\le\norm{\widehat D_s(\widehat w)-D_s(\widehat w)}
     +\norm{D_s(\widehat w)-D_s(w)}
     +\norm{\widehat D_t(y)-D_t(y)}\\
 &\le\eta_s+\alpha L_s\eta_t+\eta_t.
\end{align*}
This proves \eqref{sup:residualerror}. Global minimality gives the
identity
$P_s((s/t)y+\alpha P_t(y))=P_t(y)$ by the square expansion in \eqref{sup:resolventproof}.
Writing $\varepsilon_u=\norm{P_u-D_u}_\infty$, the same decomposition
then yields
\[
 \norm{C_{s,t}(y)}\le\varepsilon_s+(1+\alpha L_s)\varepsilon_t
 \le(2+\alpha L_s)\max\{\varepsilon_s,\varepsilon_t\}.
\]
Combining this inequality with \eqref{sup:residualerror} proves
\eqref{sup:correctedcertificate}.
\end{proof}

For a compact support of diameter $B$, one may take $L_s=B^2/(4s)$.
The certificate is positive when the observed residual exceeds the
approximation allowance $\eta_s+(1+\alpha L_s)\eta_t$.

\section{An explicit excess risk lower bound}
\label{sup:riskproof}
Assume that $K=\supp\mu$ is compact. Fix $0<s<t$, set $\alpha:=1-s/t$, and use $C_{s,t}$ from \eqref{sup:residuals}. For measurable estimators $P_u$, define
\[
\calE_u(P_u):=\E\norm{P_u(Y_u)-X}^2-\E\norm{D_u(Y_u)-X}^2
=\E\norm{P_u(Y_u)-D_u(Y_u)}^2.
\]
The equality follows from the conditional expectation projection identity \cite[Section~9.4]{williams1991probability}. For global proximal selections of a common $R$, the square expansion \eqref{sup:resolventproof} gives compatibility. With $w_P:=(s/t)y+\alpha P_t(y)$, the triangle inequality yields
\begin{equation}
\norm{C_{s,t}(y)}\le\norm{D_s(w_P)-P_s(w_P)}+(1+\alpha L_s)\norm{P_t(y)-D_t(y)},
\label{sup:localrisk}
\end{equation}
where $L_s$ is a Lipschitz constant for $D_s$. The covariance Jacobian formula and the scalar variance bound \cite{bhatia2000variance} give $L_s=(\operatorname{diam}K)^2/(4s)$.

\begin{corollary}[Positive excess risk]
\label{sup:risk}
Suppose that $K=\supp\mu$ is compact and not a singleton.
Choose $M>0$ such that $K\subseteq\mathbb B(0,M)$, and let
$B:=\operatorname{diam}K$.
Choose an input $y_0\in\R^d$ with
$\delta:=\norm{C_{s,t}(y_0)}>0$.
Define
\begin{equation}
 \begin{aligned}
 L&:=\frac{B^2}{4s},
 ~~
 a:=\frac{\delta}{2+\alpha L},\\
 b&:=\max\left\{
 \norm{y_0},
 \norm{\frac{s}{t}y_0+\alpha D_t(y_0)}+\alpha a
 \right\}.
 \end{aligned}
 \label{sup:riskconstants}
\end{equation}
Here $L$ is a common Lipschitz bound for $D_s$ and $D_t$;
the positive constants $a$ and $b$ will bound the size of an
estimation error and the norm of an input where it occurs.

For every proper lower semicontinuous $R$ and measurable
selections $P_u(y)\in\Prox_{uR}(y)$ defined at all inputs
for $u\in\{s,t\}$,
\begin{equation}
 \max\{\calE_s(P_s),\calE_t(P_t)\}
 \geq
 \frac{\omega_d a^{d+2}}
 {16^{d+1}L^d(2\pi t)^{d/2}}
 \exp\!\left[
 -\frac{(b+M+a/(4L))^2}{2s}
 \right]
 >0,
 \label{sup:riskbound}
\end{equation}
where $\omega_d$ is the $d$-dimensional volume of the
Euclidean unit ball.
\end{corollary}

\begin{proof}
Adding the global minimality inequalities at two inputs
$x,y\in\R^d$ gives
\[
 \langle P_u(x)-P_u(y),x-y\rangle\geq0.
\]
Thus each selection $P_u$ is monotone, even when $R$ is nonconvex.

We first locate an error of size at least $a$ at an input of
norm at most $b$.
If $\norm{P_t(y_0)-D_t(y_0)}\geq a$, take $u=t$ and $x_0=y_0$.
Otherwise, set
\[
 w_P:=\frac{s}{t}y_0+\alpha P_t(y_0).
\]
By \eqref{sup:localrisk} and $\delta=(2+\alpha L)a$,
\[
 \norm{P_s(w_P)-D_s(w_P)}
 \geq
 \delta-(1+\alpha L)\norm{P_t(y_0)-D_t(y_0)}
 >a.
\]
Moreover,
\[
 \norm{w_P}
 \leq
 \norm{\frac{s}{t}y_0+\alpha D_t(y_0)}
 +\alpha a
 \leq b.
\]
In this case, take $u=s$ and $x_0=w_P$.
In either case,
\[
 \norm{x_0}\leq b,
 ~~
 \norm{P_u(x_0)-D_u(x_0)}\geq a.
\]

For this choice of $u$ and $x_0$, write $P:=P_u$ and $D:=D_u$.
Define the unit error direction $v$ and the radius $r$ by
\[
 v:=
 \frac{P(x_0)-D(x_0)}{\norm{P(x_0)-D(x_0)}},
 ~~
 r:=\frac{a}{4L}.
\]
For any displacement $h\in\R^d$ satisfying
$0<\norm{h}\leq r$ and $v^\top h\geq\norm{h}/2$,
monotonicity of $P$ and the Lipschitz bound for $D$ imply
\[
 \begin{aligned}
 \left\langle
 P(x_0+h)-D(x_0+h),\frac{h}{\norm{h}}
 \right\rangle
 &\geq
 \left\langle
 P(x_0)-D(x_0),\frac{h}{\norm{h}}
 \right\rangle
 -L\norm{h}\\
 &\geq \frac a2-Lr
 =\frac a4.
 \end{aligned}
\]
These displacements include the ball
$\mathbb B(3rv/4,r/4)$.
Indeed, every $h$ in this ball has
$\norm{h}\leq r$ and $v^\top h\geq r/2\geq\norm{h}/2$.
Therefore,
\[
 \norm{P(y)-D(y)}\geq\frac a4
 ~~
 \text{for all }
 y\in\mathbb B(x_0+3rv/4,r/4).
\]

Let $p_u$ denote the density of $Y_u$.
Every point $y$ in this last ball satisfies
$\norm{y}\leq b+r$.
Since $K\subseteq\mathbb B(0,M)$ and $u\in\{s,t\}$,
\[
 \begin{aligned}
 p_u(y)
 &=(2\pi u)^{-d/2}
   \int_K\exp\!\left[-\frac{\norm{y-x}^2}{2u}\right]\mu(dx)\\
 &\geq
 (2\pi t)^{-d/2}
 \exp\!\left[-\frac{(b+r+M)^2}{2s}\right].
 \end{aligned}
\]
Integrating the squared error over this ball yields
\[
 \calE_u(P_u)
 \geq
 \frac{a^2}{16}\,
 \omega_d\left(\frac r4\right)^d
 (2\pi t)^{-d/2}
 \exp\!\left[-\frac{(b+r+M)^2}{2s}\right].
\]
Substitution of $r=a/(4L)$ proves \eqref{sup:riskbound}.
\end{proof}

\section{Gaussian characterization from an accumulation of variances}
\label{sup:gaussian}
We now allow $\mu$ to be an arbitrary probability measure on $\R^d$ and
use the kernel definition \eqref{sup:denoiser}. The noise has covariance $tI_d$. Here $JD_t$ denotes the Jacobian in the observation variable, and $\Delta:=\sum_{j=1}^d\partial_{y_j}^2$ is the Laplacian, acting componentwise on vectors. For symmetric matrices, $A\preceq B$ means that $B-A$ is positive semidefinite.

\begin{theorem}[Gaussian characterization]
\label{sup:gaussiantheorem}
Let $S\subset(0,\infty)$ have an accumulation point in $(0,\infty)$.
If a proper lower semicontinuous function $R$ satisfies
\begin{equation}
 D_t(y)\in\Prox_{tR}(y)
 ~~(t\in S,\ y\in\R^d),
 \label{sup:commonfamily}
\end{equation}
then $\mu$ is a Gaussian measure, possibly with singular covariance.
Conversely, every Gaussian measure admits such a common penalty for
all $t>0$.
\end{theorem}
\begin{proof}
We first establish the smooth identities needed for the argument.
On any compact subset of $(0,\infty)\times\R^d$, the Gaussian kernel,
its derivatives, and their products with any fixed polynomial in $x$
are uniformly bounded as functions of $x$. Since $\mu$ is finite,
differentiation under the integral in \eqref{sup:density} and
\eqref{sup:denoiser} is valid to every order. Thus $p_t$ and $D_t$ are
smooth, and the positive denominator gives finite posterior moments
at every $(t,y)$.

Write $\ell_t(y)=\log p_t(y)$, and let $\pi_{t,y}$ be the probability
measure obtained by normalizing the Gaussian weighted measure in
\eqref{sup:denoiser}. Differentiating in $y$ gives Tweedie's identity \cite{efron2011tweedie} and the covariance formula
\begin{equation}
 D_t(y)=y+t\nabla\ell_t(y),~~
 JD_t(y)=\frac1t\Cov_{\pi_{t,y}}(X)\succeq0.
 \label{sup:tweedie}
\end{equation}
All entries of this posterior covariance are finite. Direct differentiation
of the kernel also gives
\begin{equation}
 \partial_t p_t=\tfrac12\Delta p_t,~~
 \partial_t\ell_t=\tfrac12\big(\Delta\ell_t+\norm{\nabla\ell_t}^2\big).
 \label{sup:heat}
\end{equation}
Combining \eqref{sup:tweedie} and \eqref{sup:heat} yields
\begin{equation}
 \partial_tD_t(y)-\frac1t JD_t(y)(D_t(y)-y)
 =\frac t2\nabla\Delta\ell_t(y).
 \label{sup:defect}
\end{equation}
Indeed, differentiating $D_t=y+t\nabla\ell_t$ in $t$ gives
$\nabla\ell_t+(t/2)\nabla\Delta\ell_t
+t\nabla^2\ell_t\nabla\ell_t$, and
$t^{-1}JD_t(D_t-y)=(I_d+t\nabla^2\ell_t)\nabla\ell_t$.

We next extract a differential identity from the common penalty.
For any $s<t$ in $S$, set $z=D_t(y)$ and
$w=(s/t)y+(1-s/t)z$. With
$Q_u(x;q)=R(x)+\norm{x-q}^2/(2u)$, expansion gives
\begin{equation}
 Q_s(x;w)-Q_s(z;w)
 =Q_t(x;y)-Q_t(z;y)
 +\left(\frac1{2s}-\frac1{2t}\right)\norm{x-z}^2.
 \label{sup:resolventproof}
\end{equation}
The first difference on the right is nonnegative, and the last term is
strictly positive for $x\ne z$. Hence $z$ is the unique minimizer at
$(s,w)$, and \eqref{sup:commonfamily} implies
\begin{equation}
 D_s\big((s/t)y+(1-s/t)D_t(y)\big)=D_t(y).
 \label{sup:compatibility}
\end{equation}

Let $t_*>0$ be an accumulation point of $S$, and choose distinct
$s_n<t_n$ in $S$ with $s_n,t_n\to t_*$. Fix $y\in\R^d$, put
$z_n=D_{t_n}(y)$, and define
\[
 q_n(u)=z_n+\frac{u}{t_n}(y-z_n),~~
 H_n(u)=D_u(q_n(u)),~~ s_n\le u\le t_n.
\]
Equation \eqref{sup:compatibility} gives
$H_n(s_n)=H_n(t_n)=z_n$. Therefore
\begin{equation}
 0=\frac1{t_n-s_n}\int_{s_n}^{t_n}
 \left[\partial_uD_u(q_n(u))
 +JD_u(q_n(u))\frac{y-z_n}{t_n}\right]du.
 \label{sup:shrinkinginterval}
\end{equation}
As $n\to\infty$, $z_n\to D_{t_*}(y)$ and $q_n(u)\to y$ uniformly
on the shrinking interval. Smoothness makes the integrand converge
uniformly, so
\begin{equation}
 \partial_tD_{t_*}(y)+\frac1{t_*}JD_{t_*}(y)(y-D_{t_*}(y))=0.
 \label{sup:limitidentity}
\end{equation}
Together with \eqref{sup:defect}, this gives
$\nabla\Delta\log p_{t_*}=0$ on $\R^d$.

Set $U=-\log p_{t_*}$. Its Laplacian is a constant, say $k$, while
\eqref{sup:tweedie} implies
\begin{equation}
 \Delta U=k,~~
 \nabla^2U=\frac1{t_*}(I_d-JD_{t_*})\preceq\frac1{t_*}I_d.
 \label{sup:hessianbounds}
\end{equation}
If $\lambda_1,\ldots,\lambda_d$ are the Hessian eigenvalues, then
$\lambda_i\le1/t_*$ and $\sum_i\lambda_i=k$, whence
$\lambda_i\ge k-(d-1)/t_*$. Every Hessian entry is consequently bounded
on $\R^d$. Each is also harmonic, since
\[
 \Delta(\partial_i\partial_jU)
 =\partial_i\partial_j(\Delta U)=0.
\]
The Liouville theorem for bounded harmonic functions
\cite[Theorem~2.1]{axler2001harmonic} makes these entries constant.
For $d=1$ this follows directly from $U''=k$.
Thus $U$ is a quadratic polynomial. Integrability of $e^{-U}$ forces
its quadratic part to be positive definite: after orthogonal
diagonalization, a zero or negative eigenvalue would make the
integral along that coordinate infinite. Hence $p_{t_*}$ is a
Gaussian density with some mean $m$ and positive definite covariance $C$.

For the characteristic function $\widehat\mu(\xi):=\int e^{i\xi^\top x}\mu(dx)$, where $\xi\in\R^d$ and $i^2=-1$, independence of $X$ and $Z$ gives
\begin{equation}
 \widehat\mu(\xi)e^{-t_*\norm\xi^2/2}
 =e^{i m^\top\xi-\xi^\top C\xi/2},~~
 \widehat\mu(\xi)=
 e^{i m^\top\xi-\xi^\top(C-t_*I_d)\xi/2}.
 \label{sup:characteristic}
\end{equation}
Since $|\widehat\mu(\xi)|\le1$ for every $\xi$, the matrix
$\Sigma=C-t_*I_d$ is positive semidefinite. The last expression is the
characteristic function of $N(m,\Sigma)$. Applied to each linear
projection of $X$, scalar uniqueness of characteristic functions
\cite[Section~16.6]{williams1991probability} shows that every such
projection is Gaussian with the stated mean and variance. Thus
$\mu=N(m,\Sigma)$.

Conversely, suppose $\mu=N(m,\Sigma)$ with $\Sigma\succeq0$. Let
$\Sigma^\dagger$ be its Moore--Penrose inverse and set
\begin{equation}
 R(x)=\tfrac12(x-m)^\top\Sigma^\dagger(x-m)
      +\iota_{m+\operatorname{range}\Sigma}(x).
 \label{sup:gaussianpenalty}
\end{equation}
Minimizing in an orthonormal eigenbasis of $\Sigma$ gives, for every
$t>0$,
\begin{equation}
 \Prox_{tR}(y)=
 \left\{m+\Sigma(\Sigma+tI_d)^{-1}(y-m)\right\}
 =\{D_t(y)\}.
 \label{sup:gaussianprox}
\end{equation}
This includes a point mass when $\Sigma=0$.
\end{proof}

The differential identity also gives the local form of the finite
residual. For fixed $(t,y)$ and $h\downarrow0$, Taylor expansion yields
\begin{equation}
 C_{t-h,t}(y)
 =-h\left[\partial_tD_t(y)-\frac1tJD_t(y)(D_t(y)-y)\right]+O(h^2)
 =-\frac h2\Delta D_t(y)+O(h^2),
 \label{sup:smallgap}
\end{equation}
since $\Delta D_t=t\nabla\Delta\ell_t$. Thus nearby variances measure
the same differential compatibility condition that appears in
\eqref{sup:limitidentity}.

\section{Numerical methods and results}
\label{sup:numerics}
The experiments evaluate the change of the canonical penalty, the
finite residual, and excess risk. They use exact finite sums or
deterministic quadrature. All plotted values are supplied as CSV files,
and the script \path{experiments/code/run_experiments.py} reproduces
the figure and the numerical tables. No external dataset is required.

\subsection{Priors and evaluation paths}
Table~\ref{sup:priors} specifies the distributions. In the compact cases,
the endpoint $a$ maximizes the Euclidean norm over the support. For the
Gaussian control, $a=1$ specifies an ordinary mean coordinate.
\begin{table}[ht]
\centering
\caption{Distributions used in the numerical experiments.}
\label{sup:priors}
\begin{tabular}{@{}lcc@{}}
\toprule
Label & Distribution & Endpoint $a$\\
\midrule
Binary & $\Pr(X=-1)=\Pr(X=1)=1/2$ & $1$\\
Four points & $(-1.2,-0.1,0.6,1.8)$ with masses $(.12,.38,.29,.21)$ & $1.8$\\
Triangle & $(-1,-.4),(1.1,-.3),(.2,1.3)$ with masses $(.2,.5,.3)$ & $(.2,1.3)$\\
Uniform & Uniform on $[-1,1]$ & $1$\\
Gaussian & $N(0,1)$ & $1$\\
\bottomrule
\end{tabular}
\end{table}

The first experiment fixes $s=2$ and $t=3$ and follows
$z_\lambda=(1-\lambda)\E X+\lambda a$ at 100 equally spaced values of
$\lambda\in[0,0.99]$. For each $u\in\{s,t\}$, solve
$\nabla A_u(\theta)=z_\lambda$ and evaluate
$\calR_u(z_\lambda)=\theta^\top z_\lambda-A_u(\theta)
-\norm{z_\lambda}^2/(2u)$.
Finite priors use normalized exponentials evaluated after subtracting
the largest exponent. Scalar inversion uses a bracketed root solver;
the triangular prior uses the covariance as the Jacobian in the mean
equation. Uniform integrals use Gauss--Legendre quadrature with 256 nodes
on $[-1,1]$. Repeating both the penalty and variance evaluations at all
100 mean coordinates with 128 and 512 nodes checks the quadrature
accuracy. The largest change from 256 to 512 nodes, including the
residual evaluations below, is $4.690\times10^{-13}$. These refinement
differences assess numerical stability; they are not the uniform
denoiser error bounds $\eta_u$ in Section~\ref{sup:approximation}.

At $\lambda\in\{0,0.5,0.9,0.99\}$ for all five priors, the script
also compares the direct difference $\calR_2-\calR_3$ with
\eqref{sup:timeidentity}. The variance integral uses adaptive quadrature
with absolute and relative tolerances $2\times10^{-12}$.
A centered difference with increment $2\times10^{-4}$ checks
$\partial_s\calR_s=-V_s/(2s^2)$.
For the Gaussian control, all quantities have explicit forms:
\begin{equation}
 D_u(y)=\frac{y}{1+u},~~ V_u(z)=\frac{u}{1+u},~~
 \calR_u(z)=\frac{z^2}{2}+\frac12\log(1+1/u).
 \label{sup:gaussiancontrol}
\end{equation}
Thus $\calR_2-\calR_3=\tfrac12\log(9/8)$ is constant. This additive normalization does not affect minimizers.
The computed compact-prior differences are positive and become small near the sampled endpoints, consistent with the boundary limit in the main article. The four point example also shows that the
difference need not decrease monotonically along the entire path.

The second experiment fixes $s=2$ and uses 100 equally spaced variances
$t\in[2.02,6]$. In dimension $d$, the prescribed observation is
\begin{equation}
 y=t\operatorname{arctanh}(1/2)\frac{(1,\ldots,1)^\top}{\sqrt d}.
 \label{sup:witnessrule}
\end{equation}
The script evaluates $C_{2,t}(y)$ and the uniform error lower bound
$\norm{C_{2,t}(y)}/[2+(1-2/t)L_2]$.
For each compact prior, $L_2=B^2/8$ uses its support diameter;
the Gaussian uses its exact Lipschitz constant $L_2=1/3$.
Table~\ref{sup:witnesses} gives the values at $t=3$.
Here $\operatorname{arctanh}$ denotes the inverse hyperbolic tangent, and $(1,\ldots,1)^\top/\sqrt d$ is a unit vector. For the binary prior, the chosen input gives $D_t(y)=1/2$. The inputs are fixed by \eqref{sup:witnessrule}, so the table records
evaluated witnesses. Across the 100 variance pairs, the Gaussian
residual is at most $1.111\times10^{-16}$.
\begin{table}[ht]
\centering
\caption{Residuals and uniform error bounds at the prescribed inputs, $s=2,t=3$.}
\label{sup:witnesses}
\begin{tabular}{@{}lrrr@{}}
\toprule
Prior & $L_2$ & $\norm{C_{2,3}(y)}$ & Uniform error lower bound\\
\midrule
Binary & $0.5000$ & $0.059867024$ & $0.027630934$\\
Four points & $1.1250$ & $0.004537031$ & $0.001910329$\\
Triangle & $0.5525$ & $0.031124423$ & $0.014250022$\\
Uniform & $0.5000$ & $0.004749003$ & $0.002191848$\\
Gaussian & $1/3$ & $0$ & $0$\\
\bottomrule
\end{tabular}
\end{table}

\subsection{A closed form binary witness}
For $X=\pm a$, where $a>0$, with equal probabilities, the posterior mean is
$D_t(y)=a\tanh(ay/t)$. On $|z|<a$,
\begin{equation}
 \calR_t(z)=h(z/a)+\frac{a^2-z^2}{2t},~~
 h(r)=\frac{(1+r)\log(1+r)+(1-r)\log(1-r)}{2}.
 \label{sup:binarypenalty}
\end{equation}
Use $0\log0=0$ at the endpoints and extend the penalty by $+\infty$
outside $[-a,a]$. Direct subtraction gives
$\calR_s(z)-\calR_t(z)=(1/(2s)-1/(2t))(a^2-z^2)$.
With $v=ay/t$ and $c=a^2(1/s-1/t)>0$,
\begin{equation}
 C_{s,t}(y)=a\{\tanh(v+c\tanh v)-\tanh v\}.
 \label{sup:binaryresidual}
\end{equation}
For $y=(t/a)\operatorname{arctanh}(1/2)$, set $b=\tanh(c/2)$.
The addition formula for the hyperbolic tangent yields
\begin{equation}
 \delta=\norm{C_{s,t}(y)}=\frac{3ab}{4+2b},~~
 \max_{u\in\{s,t\}}\norm{P_u-D_u}_\infty
 \geq\frac{3a\tanh(c/2)}{[4+2\tanh(c/2)](2+c)}.
 \label{sup:binarycertificate}
\end{equation}
The binary row of Table~\ref{sup:witnesses} follows from this formula
with $a=1,s=2,t=3$.

\subsection{Excess risk when a binary penalty is reused}
For the equal binary prior on $\{-1,1\}$, consider the family
\begin{equation}
 R_\tau(z):=\begin{cases}h(z)-z^2/(2\tau),&|z|\le1,\\+\infty,&|z|>1,\end{cases}
 ~~ 2\leq\tau\leq3,
 \label{sup:riskfamily}
\end{equation}
with $h$ from \eqref{sup:binarypenalty}. Its members are convex, since
$h''(z)=1/(1-z^2)\geq1$ on $(-1,1)$.
The constant omitted from $\calR_\tau$ has no effect on minimizers.
At variance $u\in\{2,3\}$, denote the unique proximal minimizer by $P_{u,\tau}(y)$. It is $z=\tanh v$, where
\begin{equation}
 v+(1/u-1/\tau)\tanh v=y/u.
 \label{sup:proxroot}
\end{equation}
The derivative of the left side is at least $5/6$. We solve this scalar
equation by bracketing, with absolute tolerance $10^{-13}$ and relative
tolerance $10^{-14}$.

The excess risk is the expectation of
$(P_{u,\tau}(Y_u)-D_u(Y_u))^2$. The distribution of $Y_u$ is the equal mixture of
$N(-1,u)$ and $N(1,u)$. Gauss--Hermite quadrature with 320 nodes per
component evaluates this expectation at 101 equally spaced values of
$\tau\in[2,3]$. Repeating all evaluations with 160 nodes changes the
risks by at most $3.597\times10^{-16}$.
A scalar root gives the intersection of the two risk curves at
$\tau=2.41699607508038$. Independent adaptive integration over the
real line checks both endpoints and this intersection, with absolute
tolerance $2\times10^{-13}$ and relative tolerance $2\times10^{-11}$.
The largest discrepancy is $6.723\times10^{-17}$.
\begin{table}[ht]
\centering
\caption{Excess mean squared errors in the family \eqref{sup:riskfamily}.}
\label{sup:risks}
\begin{tabular}{@{}lcc@{}}
\toprule
$\tau$ & Variance $2$ & Variance $3$\\
\midrule
$2$ & $0$ & $2.580773094\times10^{-3}$\\
$2.41699607508038$ & $5.628146618\times10^{-4}$ & $5.628146618\times10^{-4}$\\
$3$ & $2.002533732\times10^{-3}$ & $0$\\
\bottomrule
\end{tabular}
\end{table}

The zeros in Table~\ref{sup:risks} are exact at the matched variance;
the corresponding numerical values are below $1.4\times10^{-33}$.
The intersection balances the two excess risks within
\eqref{sup:riskfamily}. For the binary witness, substitute $s=2$, $t=3$, $d=1$, $M=1$, $L=1/2$, $\alpha=1/3$, and $\omega_1=2$ into \eqref{sup:riskbound}. The input is $y_0=3\operatorname{arctanh}(1/2)$, and the constants in \eqref{sup:riskconstants} are $a\approx0.027630934$ and $b=y_0\approx1.647918433$. The resulting universal bound is $1.2916082364\times10^{-8}$, substantially below the larger risks observed in the family $R_\tau$.

\subsection{Verification data}
Table~\ref{sup:checks} records the discrepancies from the experiment
script. Its JSON output includes the selected observations, posterior
means, transformed observations, relative entropy bound slacks, and
environment versions. The calculations used Python~3.12.14,
NumPy~2.3.5, SciPy~1.17.0, and Matplotlib~3.10.8.
\begin{table}[ht]
\centering
\caption{Numerical consistency checks for the reported experiments.}
\label{sup:checks}
\begin{tabular}{@{}lc@{}}
\toprule
Comparison & Largest absolute discrepancy\\
\midrule
Variance integral and canonical penalty difference & $4.227\times10^{-15}$\\
Centered derivative and $-V_s/(2s^2)$ & $1.510\times10^{-9}$\\
Mean inversion at all plotted mean coordinates & $1.633\times10^{-13}$\\
Uniform quadrature, 256 versus 512 nodes & $4.690\times10^{-13}$\\
Risk quadrature, 160 versus 320 nodes & $3.597\times10^{-16}$\\
Risk quadrature and independent adaptive integration & $6.723\times10^{-17}$\\
Gaussian residual across all 100 variance pairs & $1.111\times10^{-16}$\\
\bottomrule
\end{tabular}
\end{table}

The additional script \path{verification/code/verify_extensions.py}
checks four symbolic identities, fixed mean comparisons for eight
finite priors, 15 relative entropy identities, 48 boundary path points,
and 2560 local certificate inequalities. Its common penalty controls
include convex quadratics and indicators of finite sets, whose global
proximal selections can be discontinuous. Gaussian controls also
include singular covariances and point masses.

\end{document}